\documentclass[11pt,reqno]{amsart}
\usepackage[utf8]{inputenc}
\usepackage[british]{babel}
\usepackage{amsmath,amssymb,amsthm,bbm,mathtools,calc,verbatim,enumitem,tikz,url,mathrsfs,cite,fullpage}
\usepackage[hidelinks]{hyperref}
\usepackage[numbers,longnamesfirst]{natbib}
\usepackage{cleveref,xspace}

\newcommand*{\ie}{i.e.\@\xspace}
\newcommand*{\bG}{\mathbf{G}}
\newcommand*{\diam}{\operatorname{diam}\xspace}

\newcommand{\N}{\mathbb{N}}
\newcommand{\pr}{\mathbb{P}}
\renewcommand{\leq}{\leqslant}
\renewcommand{\geq}{\geqslant}

\newtheorem{thm}{Theorem}[section]
\newtheorem{lem}[thm]{Lemma}

\theoremstyle{definition}
\newtheorem{remark}{Remark}

\crefname{thm}{theorem}{theorems}
\Crefname{thm}{Theorem}{Theorems}
\crefname{lem}{lemma}{lemmas}
\Crefname{lem}{Lemma}{Lemmas}
\crefname{defi}{definition}{definitions}
\Crefname{defi}{Definition}{Definitions}
\crefname{cor}{corollary}{corollaries}
\Crefname{cor}{Corollary}{Corollaries}
\crefname{prop}{proposition}{propositions}
\Crefname{prop}{Proposition}{Propositions}
\crefname{remark}{remark}{remarks}
\Crefname{remark}{Remark}{Remarks}

\begin{document}

\title{The diameter of randomly twisted hypercubes}
\author{Lucas Arag\~ao \and Maur\'icio Collares \and Gabriel Dahia \and Jo\~ao Pedro Marciano}

\address{UERJ, S\~ao Francisco Xavier 524, Rio de Janeiro, RJ, Brazil}
\email{lucas.aragao@uerj.br}

\address{IMPA, Estrada Dona Castorina 110, Jardim Bot\^anico, Rio de Janeiro, RJ, Brazil}
\email{\{gabriel.dahia, joao.marciano\}@impa.br}

\address{Institute of Discrete Mathematics, Graz University of Technology, Steyrergasse 30, 8010 Graz, Austria}
\email{mauricio@collares.org}

\begin{abstract}
  The $n$-dimensional random twisted hypercube $\bG_n$ is constructed recursively by taking two instances of $\bG_{n-1}$, with any joint distribution, and adding a random perfect matching between their vertex sets.
  \citet*{benjamini2022duplicube} showed that its diameter is $O(n\log \log \log n/\log \log n)$ with high probability and at least ${(n - 1)/ \log_2 n}$.
  We improve their upper bound by showing that
  $$\diam(\bG_n) = \big(1 + o(1)\big) \frac{n}{\log_2 n}$$
  with high probability.
\end{abstract}

\maketitle

\section{Introduction}

The random twisted hypercube $\bG_n$ is a model for generating $n$-regular graphs with $2^n$ vertices.
\citet*{benjamini2022duplicube} introduced it by generalising a related model studied by~\citet*{dudek2018randomly}.
It is defined recursively in a similar way as the hypercube $Q_n$: $\bG_1$ is a single edge and
$\bG_{n}$ is defined by adding a perfect matching $M_n$ between two copies of
$\bG_{n-1}$. However, rather than matching vertices according to a (fixed) vertex ordering, as in $Q_n$, or even deterministically,
the matching $M_n$ is chosen uniformly at random and independently in each step.
Moreover, instead of using two equal copies of $\bG_{n-1}$ in the inductive construction,
we allow the pair to be sampled with an arbitrary joint distribution. Two concrete
examples are the \emph{random duplicube} and the aforementioned \emph{independent twisted hypercube} studied
by~\citet*{dudek2018randomly}\footnote{The authors of \cite{dudek2018randomly} use the terminology ``random twisted hypercube'' for the special case where the two copies of $\bG_{n-1}$ are independent; however, we shall instead follow the terminology of~\cite{benjamini2022duplicube}.}. In the former, we sample a single graph from $\bG_{n-1}$ and use two copies of it to construct $\bG_{n}$, where, in the latter, the two copies of $\bG_{n-1}$ are chosen independently.

One of the motivations for studying the random twisted hypercube is that it can be understood as a middle
ground between the structure of $Q_n$, the $n$-dimensional hypercube, and $H_n$, the random $n$-regular graph on $2^n$ vertices.
To support this intuition, \citeauthor*{benjamini2022duplicube} studied several different properties of $\bG_n$, including its eigenvalue distribution, vertex expansion and diameter; here we will focus on the diameter of $\bG_n$.
Recall that the diameter of $H_n$ is $(1+o(1))\frac{n}{\log_2 n}$ with high probability \citep{bollobas1982diameter} and that the diameter of $Q_n$ is $n$.
The same simple greedy argument that proves the upper bound for $Q_n$ works for $\bG_n$, but it
was shown in \cite{benjamini2022duplicube} that this trivial bound is typically far from
the truth, and in fact $\diam(\bG_n)$ is sublinear in $n$ with high probability.
The authors of \citep{benjamini2022duplicube} also observed that the deterministic lower bound on $\diam(H_n)$ also holds for $\bG_n$.
To be precise, they showed that the bounds
\[\frac{n-1}{\log_2 n}\leq\diam(\bG_n)\leq C\frac{n \log\log\log n}{\log\log n}\]
hold with high probability for some constant $C > 0$.
In the particular case of the random duplicube, they showed that $\diam(\bG_n)=(1+o(1))n/\log_2 n$ with high probability, matching a result of~\cite{dudek2018randomly} for the independent twisted hypercube.
Moreover, they remarked that they ``have no intuition to the correct diameter'' in the general case where the two samples of $\bG_{n-1}$ used to construct $\bG_{n}$ may have arbitrary joint distribution.
In this note, we asymptotically determine the diameter of $\bG_n$.

\begin{thm}\label{thm:diam}
  Let $\bG_n$ be the random twisted hypercube, with arbitrary joint distribution. Then
  \[\diam(\bG_n) = (1+o(1))\frac{n}{\log_2 n}\]
  with high probability as $n\to \infty$.
\end{thm}

We now give a brief sketch of the proof of \Cref{thm:diam}, starting with a more naive
approach, and then explaining how to improve it. In what follows, all logarithms are base
$2$. Identify the vertices of $\bG_n$ with $\{0,1\}^n$ so that every induced subgraph of
$\bG_n$ that is given by a set of vertices with the same final $n-k$ coordinates is
distributed as $\bG_k$. Let $u,v \in V(\bG_n)$ be distinct, arbitrary vertices that we want to
connect using as few edges as we can and define $\alpha(u,v)$ to be the
largest coordinate in which they differ. Note first that $d(u,v) \leq \alpha(u,v)$ by the
following greedy approach: $u$ has a (unique) neighbour $w$ with
$\alpha(w,v)\leq\alpha(u,v) - 1$, so we may add the edge $uw$ to our path and repeat with
$u$ replaced by $w$ to finish in at most $\alpha(u,v)$ steps, \ie we reach $v$ when
$\alpha(w,v) = 0$.

The key idea to prove \Cref{thm:diam} is to use a less greedy approach.
Let $M_k(u)$ denote
the random matching between the distinct copies of $\bG_{k-1}$ with final coordinates $(u_k,\ldots,u_n)$ and $(1-u_k,u_{k+1},\ldots,u_n)$, and let $t = t(n)$ be a slowly growing function of $n$.
In each step, we shall choose an edge $u'v'$ of $M_k(u)$ for $k = \alpha(u, v)$ such that $d(u,u') \leq t$ and $\alpha(v',v)$ is as small as possible.
If $\alpha(v',v)\geq n/(\log n)^2$, we repeat the process with $v'$ in place of $u$.
Otherwise, since $d(v',v)\leq \alpha(v',v)$ by the previous naive approach, we may finish the process by connecting $v'$ to $v$.
The number of steps taken, and therefore the length of the constructed path, will depend on how much $\alpha(u,v)$ decreases when we replace $u$ by $v'$.
Since $u'v'$ is chosen to be an edge of $M_k(u)$, the existence of a ``good'' choice of $v'$ depends on the size of the set $B$ of vertices of $\bG_n[\{w:\alpha(w,u) < k\}]$ within distance $t$ of $u$.

This idea motivates the following ``quasirandomness'' property: for every $k \geq n / (\log n)^2$ and every pair of vertices $u, v$ such that $\alpha(u, v) = k$,
there is a vertex $v' \in V(\bG_n)$ at distance at most $t+1$ from $u$ such that
$\alpha(v',v) \leq k - (t - 2)\log n$. To show that it holds for $\bG_n$, we will use a union bound over pairs of vertices. Fix $u$ and $v$ with $\alpha(u,v) = k$, and reveal the two samples of $\bG_{k-1}$ that contain $u$ and $v$. We then use the randomness of the matching $M_k(u)$ (and its independence of the choice of each $\bG_{k-1}$) to bound the probability that such a vertex $v'$ does not exist. To prove this, we first show (see \Cref{lem:balls}) that, deterministically,
$|B| \geq  (k/t)^t$. Since the neighbours of $B$ in $M_k(u)$ form a random subset of the sample of $\bG_{k-1}$ that contains $v$, we can deduce (see \Cref{lem:whp}) that a suitable $v'$ exists with probability at least $1-\exp(-n^{3/2})$.

The above quasirandomness property is useful because it implies that the claimed upper bound on the
diameter holds \emph{deterministically}. As at each step we grow our path by at most
$t+1$ and $\alpha(u,v) \leq n$, we take at most $n/((t - 2)\log n)$ steps to reach a
vertex that can be greedily connected to $v$ using at most $n / (\log n)^2$ edges. Since
$t = \omega(1)$, the resulting path has length at most \[(t+1) \frac{n}{(t - 2)\log n} +
\frac{n}{(\log n)^2} = (1+o(1))\frac{n}{\log n}.\] The rest of the note is dedicated to
formalising this proof.

\section{The proof}\label{sec:proof}

The \emph{$n$-dimensional random twisted hypercube} $\bG_n$ is defined recursively by placing a random matching between two instances of $\bG_{n-1}$.
Formally, we set $\bG_1=K_2$ and $M_1=E(K_2)$, and for every $n \geq 2$ we define $\bG_n$ in the $n$-th recursive step as follows.
We fix a distribution on pairs of graphs with both marginal distributions equal to $\bG_{n-1}$; that is, a pair $(G_{n-1}^{(0)}, G_{n-1}^{(1)})$ sampled according to this distribution satisfies $\mathbb{P}(G_{n-1}^{(i)} = G) = \mathbb{P}(\bG_{n-1} = G)$ for every graph $G$ and every $i \in \{0,1\}$.
We then sample $(G_{n-1}^{(0)}, G_{n-1}^{(1)})$ according to this distribution, and choose a perfect matching $M_n$ between $V(G_{n-1}^{(0)})$ and $V(G_{n-1}^{(1)})$ uniformly at random and independently of all other choices made so far.
Then, $\bG_n$ will be the disjoint union of $G_{n-1}^{(0)}$ and $G_{n-1}^{(1)}$ with the perfect matching $M_n$ in between, \ie,
\begin{equation*}
	V(\bG_n)=V(G_{n-1}^{(0)})\cup V(G_{n-1}^{(1)}) \qquad \text{ and } \qquad E(\bG_n)=E(G_{n-1}^{(0)})\cup E(G_{n-1}^{(1)}) \cup M_n.
\end{equation*}

We stress that the choice of this perfect matching at each step is independent of all earlier choices.
Observe we can assume that $V(\bG_n)=\{0,1\}^n$ and that for $i=0,1$ the set of vertices with final coordinate equal to $i$ is the vertex set of $G_{n-1}^{(i)}$.
For each $u \in V(\bG_n)$ and $1 \leq k < n$, the subgraph induced by the set of vertices with the same $n-k$ final coordinates as $u$ is an instance of $\bG_k$.
Therefore, $\bG_{n}$ contains exactly $2^{n-k}$ (correlated) samples of $\bG_{k}$.

Recall that, for each $v \in V(\bG_n)$, we write $M_k(v)$ to denote the (last to be generated) random matching that appears inside the copy of $\bG_{k}$ that contains $v$. It will also
be useful to index each neighbour of $v$ by the recursive step that introduced it,
so $N(v)=\big\{\eta_1(v), \ldots, \eta_n(v)\big\}$ where $\eta_k(v)$ is the neighbour of $v$
given by the matching $M_k(v)$. Therefore, by construction, $\eta_k(v)$ differs from $v$
in the $k$-th coordinate and coincides with $v$ in all coordinates greater than $k$. More
precisely, letting $u=\eta_k(v)$, we have $u_k\neq v_k$ and $u_i=v_i$ for every $k<i \leq
n$.

For ${u,v \in V(\bG_n)}$, define $\alpha(u,v)$ to be the largest coordinate in which $u$ and $v$ differ, \ie,
\begin{equation*}
  \alpha(u,v) = \left\{
    \begin{array} {c@{\quad \textup{if} \quad}l}
      \max \big\{j \in [n]: u_j \neq v_j\big\} & u \neq v \\[+1ex]
      0 & u = v.
    \end{array}\right.
\end{equation*}
Given $u, v \in V(\bG_n)$, recall that in the greedy approach we can construct a path from $u$ to $v$ of length at most $\alpha(u,v) \leq n$.
Using the new notation to describe it, we start from $u$ and iteratively replace $u$ by $\eta_{\alpha(u,v)}(u)$, that is, we use the matching $M_{\alpha(u,v)}(u)$ to ``flip'' the largest coordinate of $u$ in which it differs from $v$.

The crux of the proof of \Cref{thm:diam} is the following lemma, which shows that the
less greedy approach we outlined in the introduction is indeed more efficient. It implies
that $\bG_n$ has the quasirandomness property defined in the introduction with high
probability.

\begin{lem}\label{lem:whp}
  Let $t \geq 3$ and let $n,k \in \N$ be sufficiently large integers such that $4 t^{2t} \leq k \leq n$ and $k \geq n/(\log n)^2$.
  Let $E_{k,t}$ denote the event that for every $u,v \in V(\bG_n)$ such that $\alpha(u,v)=k$, some vertex $w$ of the graph $G' = \bG_n\big[\big\{w : w_i = u_i, \forall i \geq k \big\}\big]$ satisfies
  \[d_{G'}(u, w) \leq t \quad\text{ and }\quad \alpha(\eta_k(w), v) \leq k - (t-2) \log k.\]
  Then $\pr(E_{k,t}^c) \leq \exp(-k^{3/2})$.
\end{lem}

Define for a graph $G$ the closed ball of
radius $t \in \N$ centred at $v \in V(G)$ to be $B_{G}(v, t) = \big\{w \in V(G) :
d_{G}(v, w) \leq t\big\}$.
To prove \Cref{lem:whp} we shall need the following deterministic lemma, which says that
closed balls in $\bG_n$ are always large.

\begin{lem}\label{lem:balls}
  Let $t, n \in \N$ and $v \in V(\bG_n)$.
  Then, deterministically, \[|B_{\bG_n}(v, t)| \geq \binom{n+1}{t}.\]
\end{lem}

\begin{proof}
	We will show that the function $f \colon \binom{[n]}{\leq t} \to B_{\bG_n}(v, t)$ given by
	\[ \big\{a_1, \ldots, a_s\big\} \mapsto \big(\eta_{a_s} \circ \cdots \circ \eta_{a_1}\big)(v), \]
	where $a_1 > \cdots > a_s$, is injective.
  From this, we will conclude the proof, since
	$$ |B_{\bG_n}(v, t)| \geq \sum_{i=0}^t\binom{n}{i} \geq \binom{n}{t-1} + \binom{n}{t} = \binom{n+1}{t}.$$
	Recall that, for every $u \in V(\bG_n)$, $u$ and $\eta_k(u)$ differ on the $k$-th coordinate and agree on every coordinate of greater index.
	We will show that every sequence $n \geq a_1 > \cdots > a_s$ can be uniquely recovered from the vertex $w = \big(\eta_{a_s} \circ \cdots \circ \eta_{a_1}\big)(v)$.

	To do so, first observe that $w$ is such that $v_{a_1} \neq w_{a_1}$ and $v_i = w_i$ for all $i > a_1$.
	Indeed, every other $a_i$ with $i>1$ is strictly smaller than $a_1$ and, as observed before, $\eta_{a_i}$ does not change coordinates strictly greater than $a_i$.
	By definition of $\alpha$, we therefore have $a_1 = \alpha(v, w)$.
	Replacing $v$ by $\eta_{a_1}(v)$, we may find $a_2$ and so on.
	The procedure stops when $v = w$, at which point we recover uniquely the sequence $(a_1,\ldots,a_s)$.
\end{proof}

We now turn our attention to \Cref{lem:whp}.
Observe that for any $k \leq n$ we have
\begin{equation}\label{eq:size-of-subcube}
	\big|\big\{v' \in V(\bG_n):\alpha(v',v) \leq  k\big\}\big| =  2^{k}.
\end{equation}
In the proof of \Cref{lem:whp}, we will use this fact, together with the bound in \Cref{lem:balls}, to show that $B_{G'}(u, t)$ intersects $\big\{w \in V(\bG_n):\alpha(\eta_k(w),v) \leq  k - (t-2) \log k\big\}$ with probability at least $1-\exp(-2k^{3/2})$.

\begin{proof}[Proof of \Cref{lem:whp}]
  Let $k,t \in \N$, and define $G^*(z) = \bG_n\big[\big\{w : w_i = z_i, \forall i \geq k \big\}\big]$ for $z \in V(\bG_n)$. Fix a pair $u,v \in V(\bG_n)$ such that $\alpha(u,v)=k$, and let $G' = G^*(u)$ and $G'' =G^*(v)$ be the instances of $\bG_{k-1}$ in $\bG_n$ containing $u$ and $v$, respectively, and consider
  \[ B' = B_{G'}(u, t) \quad\text{and}\quad B'' = \big\{ w \in V(G'') : \alpha(w, v) \leq k - (t-2)\log k \big\}. \]
  By \Cref{lem:balls}, applied to $G'$, and \eqref{eq:size-of-subcube}, respectively, we have
  \[ |B'| \geq \binom{k}{t} \geq \left(\frac{k}{t}\right)^t \qquad\text{and}\qquad|B''| \ge 2^{k - 1 - (t-2)\log k}. \]
  We emphasise that both of the above equations hold deterministically, and that
  so far we have only revealed $G'$ and $G''$.

  Let $A = A(u, v)$ be the event that $\eta_k(B') \cap B''  = \emptyset$. Since the matching connecting $G'$ and $G''$ is chosen independently of $G'$ and $G''$, it follows that (conditioned on the pair $G',G''$) $\eta_k(B')$ is a uniformly-chosen random subset of $V(G'')$ of size $|B'| \geq (k/t)^t$.
  We then obtain\footnote{If $2^{k-1} - |B''| < |B'|$, then the product is zero, and the inequality holds trivially.}
  \[ \pr(A) = \prod_{i=0}^{|B'|-1} \frac{2^{k-1} - |B''| - i}{2^{k-1} - i} \leq \big(1 - k^{-(t-2)}\big)^{|B'|} \leq \exp\left(-\frac{k^2}{t^t}\right), \]
  and therefore $\pr(A) \leq \exp(-2k^{3/2})$ by the hypothesis $k \geq 4 t^{2t}$.
  Since $E_{k,t}^c$ implies that $A(u,v)$ holds for some $(u,v) \in V(\bG_n)^2$ with $\alpha(u,v) = k$, we deduce by the union bound that
  \begin{equation}\label{eq:union-bound-over-A-uv}
    \pr(E_{k,t}^c) \leq 2^{2n}\cdot \exp\big(-2k^{3/2}\big) \leq \exp\big(-k^{3/2}\big),
  \end{equation}
  where the last inequality uses the hypothesis $k \geq n/(\log n)^2$.
\end{proof}

To highlight that the non-independence of the samples of $\bG_k$ inside $\bG_n$ is not an issue, consider the following alternative way of understanding the union bound in \eqref{eq:union-bound-over-A-uv}.
The event whose probability we bound is the existence of a sample of $\bG_k$ inside $\bG_n$ with the following property: there are two vertices $u, v \in V(\bG_n)$ inside this specific sample of $\bG_k$ such that $\alpha(u, v) = k$ and $A(u,v)$ holds.

We are now ready to prove \Cref{thm:diam}.

\begin{proof}[Proof of \Cref{thm:diam}]
  As observed in \citep{benjamini2022duplicube}, the lower bound follows from observing that, in an $n$-regular graph, there are at most $2n^d$ vertices within distance $d$ from any given vertex.

  To prove the upper bound, let $n_0 = n/(\log n)^2$ and $t = \frac{\log n_0}{4 \log \log n_0}$, which satisfies $n_0 \geq 4 t^{4t}$ if $n$ is large enough.
  Moreover, let $E_{k,t}$ be the events defined in \Cref{lem:whp}. By the union bound, the event $E := \bigwedge_{k=n_0}^n E_{k,t}$ holds with probability at least $1 - \sum_{k=n_0}^n \exp(-k^{3/2}) = 1 - o(1)$.

  We now assume that $E$ holds and show that $\diam(\bG_n) = \big(1+o(1)\big)n/\log n$. To do so, we consider arbitrary $u, v \in V(\bG_n)$ and construct a sequence of vertices $u^{(0)}, \ldots, u^{(L)} \in V(\bG_n)$ satisfying
  \begin{equation}\label{eq:induction-properties}
    u^{(0)} = 0, \qquad d(u^{(i)}, u^{(i+1)}) \leq t+1 \quad\text{and}\quad \alpha(u^{(i+1)}, v) \leq \alpha(u^{(i)}, v) - (t-2)\log n_0
 \end{equation}
 for every $0 \leq i < L$, as well as satisfying $\alpha(u^{(L)}, v) < n_0$.

To construct the desired sequence, we start by setting $u^{(0)} = u$. Assuming inductively that we have constructed $u^{(0)}, \ldots, u^{(i)}$ satisfying the above properties, we do the following. If $\alpha(u^{(i)}, v) < n_0$, we set $L = i$ and finish the process. Otherwise, set $k := \alpha(u^{(i)}, v)$; by the assumption that $E_{k,t}$ holds, there exists a vertex $w$ such that
\[ d(u^{(i)}, w) \leq t\qquad\text{and}\qquad \alpha(\eta_{k}(w), v) \leq \alpha(u^{(i)}, v) - (t-2)\log n_0. \]
We may then set $u^{(i+1)} := \eta_{k}(w)$, which satisfies~\eqref{eq:induction-properties}, and continue inductively.
At the end of the iteration, we have $d(u^{(L)}, v) \leq \alpha(u^{(L)}, v) < n_0$. Therefore, by the second property in~\eqref{eq:induction-properties},
\begin{equation}\label{eq:total-dist}
  d(u^{(0)}, v) \leq d(u^{(L)}, v) + \sum_{i=0}^{L-1} d(u^{(i)}, u^{(i+1)}) \leq n_0 + (t+1)L.
\end{equation}
Moreover, by the third property,
\[ L \cdot (t-2)\log n_0 \leq \sum_{i=0}^{L-1} \big(\alpha(u^{(i)}, v) - \alpha(u^{(i+1)}, v)\big) \leq n, \]
since the sum telescopes. Rearranging to obtain an upper bound on $L$, and plugging it into~\eqref{eq:total-dist}, we deduce that
\begin{equation}\label{eq:final_eq}
	d(u,v) = d(u^{(0)}, v) \leq n_0 + \frac{t+1}{t-2} \cdot \frac{n}{\log n_0} = \big(1+o(1)\big)\frac{n}{\log n},
\end{equation}
where the last equality follows from $n_0 = o(n/\log n)$, $\log n_0 = (1+o(1))\log n$ and $t = \omega(1)$. Since $u$ and $v$ were arbitrary, the result follows.
\end{proof}

\begin{remark}
	In fact, the last step in \eqref{eq:final_eq} can be refined to obtain
	\[\diam(\bG_n)=\frac{n}{\log n}+O\left(\frac{n \log\log n}{(\log n)^2}\right)\]
	with high probability.
\end{remark}

\section{Acknowledgements}

We would like to thank Rob Morris for many suggestions that significantly improved this note.
We would also like to thank him and Joshua Erde for helpful discussions, and the anonymous referees for their comments.

This study was financed in part by the Coordenação de Aperfeiçoamento de Pessoal de Nível Superior, Brasil (CAPES), Finance Code 001.
The second author was partially supported by the Conselho Nacional de Desenvolvimento Científico e Tecnológico, Brasil (CNPq) grant 406248/2021-4.
This research was funded in whole, or in part, by the Austrian Science Fund
(FWF) P36131. For the purpose of open access, the author has
applied a CC BY public copyright licence to any Author Accepted Manuscript
version arising from this submission.

\bibliographystyle{abbrvnat}
\bibliography{bib}

\begin{thebibliography}{3}
\providecommand{\natexlab}[1]{#1}
\providecommand{\url}[1]{\texttt{#1}}
\expandafter\ifx\csname urlstyle\endcsname\relax
  \providecommand{\doi}[1]{doi: #1}\else
  \providecommand{\doi}{doi: \begingroup \urlstyle{rm}\Url}\fi

\bibitem[Benjamini et~al.(2022)Benjamini, Dikstein, Gross, and
  Zhukovskii]{benjamini2022duplicube}
I.~Benjamini, Y.~Dikstein, R.~Gross, and M.~Zhukovskii.
\newblock Randomly twisted hypercubes -- between structure and randomness,
  2022.
\newblock arXiv:2211.06988.

\bibitem[Bollob{\'a}s and Fernandez de~la Vega(1982)]{bollobas1982diameter}
B.~Bollob{\'a}s and W.~Fernandez de~la Vega.
\newblock The diameter of random regular graphs.
\newblock \emph{Combinatorica}, 2\penalty0 (2):\penalty0 125--134, 1982.

\bibitem[Dudek et~al.(2018)Dudek, P{\'e}rez-Gim{\'e}nez, Pra{\l}at, Qi, West,
  and Zhu]{dudek2018randomly}
A.~Dudek, X.~P{\'e}rez-Gim{\'e}nez, P.~Pra{\l}at, H.~Qi, D.~West, and X.~Zhu.
\newblock Randomly twisted hypercubes.
\newblock \emph{European Journal of Combinatorics}, 70:\penalty0 364--373,
  2018.

\end{thebibliography}

\end{document}